\documentclass{article}
\usepackage{pinlabel}
\usepackage{amsfonts}
\usepackage{amsmath}
\usepackage{amssymb}
\usepackage{amsthm}
\usepackage{stmaryrd}
\usepackage[all]{xy}
\usepackage{bbm}
\usepackage{hyperref}

\usepackage[citestyle=numeric,bibstyle=numeric, backend=bibtex, block=ragged]{biblatex}
\bibliography{references}

\title{Relative Thom spectra via operadic Kan extensions}
\author{Jonathan Beardsley}

\theoremstyle{definition} 
\newtheorem{rem}{Remark}

\newtheorem{defi}[rem]{Definition}

\theoremstyle{plain}
\newtheorem*{thm*}{Theorem}
\newtheorem*{cor*}{Corollary}
\newtheorem{thm}[rem]{Theorem}
\newtheorem{lem}[rem]{Lemma}
\newtheorem{prop}[rem]{Proposition}
\newtheorem{cor}[rem]{Corollary}

\newcommand{\ints}{\mathbb{Z}}

\newcommand{\Cat}{q\mathcal{C}at}
\renewcommand{\S}{\mathcal{S}}
\newcommand{\T}{\mathcal{T}}

\newcommand{\sph}{\mathbb{S}}

\newcommand{\Fin}{\mathcal{F}in_\ast}
\newcommand{\E}{\mathbb{E}}

\newcommand{\MString}{M\mathrm{String}}
\newcommand{\BString}{B\mathrm{String}}
\newcommand{\MSpin}{M\mathrm{Spin}}
\newcommand{\BSpin}{B\mathrm{Spin}}
\newcommand{\BO}{B\mathrm{O}}
\newcommand{\MO}{M\mathrm{O}}
\newcommand{\MU}{M\mathrm{U}}
\newcommand{\BU}{B\mathrm{U}}
\newcommand{\MSO}{M\mathrm{SO}}
\newcommand{\BSO}{B\mathrm{SO}}

\newcommand{\MSp}{M\mathrm{Sp}}
\newcommand{\MSU}{M\mathrm{SU}}
\newcommand{\BSU}{B\mathrm{SU}}
\newcommand{\Spin}{\mathrm{Spin}}
\newcommand{\SO}{\mathrm{SO}}
\newcommand{\U}{\mathrm{U}}
\newcommand{\SU}{\mathrm{SU}}

\begin{document}

\maketitle

\begin{abstract}
We show that a large number of Thom spectra, i.e. colimits of morphisms $BG\to BGL_1(\sph)$, can be obtained as iterated Thom spectra, i.e. colimits of morphisms $BG\to BGL_1(Mf)$ for some Thom spectrum $Mf$. This leads to a number of new relative Thom isomorphisms, e.g. $\MU[6,\infty)\wedge_{\MString} MU[6,\infty)\simeq MU[6,\infty)\wedge\sph[B^3\Spin]$. As an example of interest to chromatic homotopy theorists, we also show that Ravenel's $X(n)$ filtration of $\MU$ is a tower of intermediate Thom spectra determined by a natural filtration of $\BU$ by sub-bialagebras. 
\end{abstract}

\section{Introduction}
In this paper we prove several new results about Thom spectra which are $\E_n$-ring spectra. The most immediately accessible results are relative Thom isomorphisms like the following from Section \ref{interexs}

\begin{itemize}
\item $\MSpin\wedge_{\MString}\MSpin\simeq \MSpin\wedge\sph[K(\ints,4)]$
\item $\MSO\wedge_{\MU}\MSO\simeq \MSO\wedge \sph[\Spin]$ 
\item $\MU\wedge_{\MSp}\MU\simeq MU\wedge\sph[SO/U]$
\item $\MU[6,\infty)\wedge_{\MString} MU[6,\infty)\simeq MU[6,\infty)\wedge\sph[B^3\Spin]$
\item $H\ints/2\wedge_{H\ints}H\ints/2\simeq H\ints/2\wedge\sph[S^1]$.
\end{itemize}

 However, in consideration of the fact that there are now a number of different methods for defining such objects, we will take a moment to clarify precisely which models we use for the remainder (though we do not expect that the choice of model is relevant to the veracity of the statements). By the category of spectra, which we denote by $\S$, we will always mean Lurie's symmetric monoidal quasicategory of spectra defined in Section 1.4.3 of \cite{ha}. In general, except when we explicitly state otherwise, we will always be working with quasicategories and all of our constructions will be homotopy invariant. For example, all of our tensor products are derived (as they must be when working internally to a quasicategory), all of our limits and colimits are the quasicategorical analogs of homotopy colimits and limits, and our functors are actually morphisms of simplicial sets. We also make use of Lurie's notion of $\infty$-operads, defined and described in Chapter 2 of \cite{ha}. We will not review the theory of $\infty$-operads here except to say that they are a natural generalization of the notion of multicategories and the categories of operators of \cite{maythom}. We are especially interested in the $\E_n$ $\infty$-operads of Chapter 5 of \cite{ha}, and will refer to them frequently in this paper. These $\infty$-operads capture the same structure as Boardman and Vogt's little $n$-cubes operads, which use embeddings of $n$-dimensional cubes to parameterize multiplicative structure (cf. Chapter 4 of \cite{maygeom}). It is non-trivial to show that these quasicategorical constructions behave identically to their model category theoretic analogs, and that results obtained thereby are compatible with results obtained using model categories. The interested reader is invited to refer to \cite{htt} and \cite{ha} for proofs that these conditions are met. We recognize, of course, that these references are expansive in their own right, so will endeavor to give more specific citations throughout the paper.

Thom spectra, the main objects of investigation here, are classically constructed by considering spaces associated to stable spherical bundles on topological spaces (see e.g. \cite{sullonthom}). However, after work of May and Sigurdsson \cite{maysig} and later work of Ando, Blumberg, Gepner, Hopkins and Rezk \cite{abghr}, it became clear that there was an alternative way to think of Thom spectra: as quotients of ring spectra by group actions. In general, given an $\E_{n}$-ring spectrum $R$, there is an $n$-fold loop space of units, $GL_1(R)$. Thus a morphism of $n$-fold loop spaces $X\to GL_1(R)$ gives an action of $X$ on $R$, and induces a morphism of $(n-1)$-fold loop spaces $BX\to BGL_1(R)$. By working quasicategorically, we can see that there is in fact a functor (thinking of $BX$ and $BGL_1(R)$ as quasicategories) $BGL_1(R)\hookrightarrow LMod_R$ which is fully faithful. Thus we have a morphism of simplicial sets $BX\to LMod_R$ (the quasicategory of left $R$-modules, where $R$ is considered as an $\E_1$-ring spectrum) which is picking out an action of $X$ on $R$ by $R$-module equivalences. This determines a diagram in $LMod_R$ whose colimit, as the thing on which every point of $X$, including the identity, acts in the same way, must be exactly $R/X$.  When $R=\sph$, the sphere spectrum, $BGL_1(\sph)$ is precisely the classifying space of stable spherical fibrations, and taking the colimit of a morphism $BX\to BGL_1(\sph)\hookrightarrow LMod_{\sph}$ produces a spectrum which is equivalent to the one produced classically by taking a sequence of Thom spaces over $BX$ (cf. Proposition 3.23 of \cite{abghr}). Thus, for instance, $\MU$ is just $\sph/U$ and $\MO$ is just $\sph/O$, and so on and so forth. 

Similarly, if we have an action of a Lie group $G$ on a smooth manifold $X$, we can take its homogeneous space $X/G$. If we happen to have an inclusion of a normal subgroup $H\hookrightarrow G$, then we obtain an action of $H$ on $X$ and can also take the homogeneous space $X/H$. It is a classical fact then that $X/H$ admits an action of $G/H$ and moreover that the iterated homogeneous space $(X/H)/(G/H)$ is diffeomorphic to $X/G$ (cf. Section 1.6, Proposition 13 of \cite{bourblie}). It stands to reason then that something similar should be true for actions of $n$-fold loop spaces on a ring spectrum $R$, and that is one of the main theorems of this document (cf. Theorem \ref{mainthm}), if we allow ourselves to replace the condition ``$H$ is a normal subgroup of $G$" with ``there is a fiber sequence of $n$-fold loop spaces $H\to G\to G/H$." Specifically, if we have a $G$-action on an $\E_n$-ring spectrum $R$, then we obtain a $G/H$-action on $R/H$, and a sequence of ring spectra $R\to R/H\to (R/H)/(G/H)\simeq R/G$. In other words, $R/G$ can be produced as the Thom spectrum associated to an action map $G/H\to BGL_1(R/H)$. 

In Section \ref{its} we show that a number of classical Thom spectra over $\sph$ can in fact be constructed as Thom spectra over \emph{other} Thom spectra. This statement is made rigorous by the following theorem, which is Theorem \ref{mainthm} in the body of the paper:  

\begin{thm*}
Suppose $Y\overset{i}\to X\overset{q}\to B$ is a fiber sequence of reduced $\E_n$-monoidal Kan complexes for $n>1$ with $i$ and $q$ both maps of $\E_n$-algebras. Let $f\colon X\to BGL_1(\sph)$ be a morphism of $\E_n$-monoidal Kan complexes for $n>1$. Then there is a a morphism of $\E_{n-1}$-algebras $B\to BGL_1(M(f\circ i))$ whose associated Thom spectrum is equivalent to $Mf$. 
\end{thm*}

By constructing $Mf$ as a Thom spectrum over an intermediate Thom spectrum, we get a relative Thom isomorphism:

\begin{cor*}
There is a morphism of $\E_{n-1}$-ring spectra $R\to M(f\circ i)\to Mf$ and a relative Thom isomorphism $Mf\wedge_{M(f\circ i)} Mf\simeq Mf\wedge_R R[B]$ where $R[B]=R\wedge_\sph\Sigma^\infty_+ B$. 
\end{cor*}

The proof requires certain technical details and constructions from \cite{ha}, so we  separate the relevant lemmas into their own subsection (\ref{lemmas}) and refer to them as needed. In the final section we give a number of examples of constructions of intermediate Thom spectra which are $\E_n$-rings. The last example we present is a new construction of $\MU$ which bears some resemblance to Lazard's construction of the Lazard ring in \cite{laz}. This construction is unrelated to recent work regarding $MU$ and complex orientations by McKeown \cite{mckeown}. This paper comprises work contained in the author's doctoral thesis. 

Let us fix some notation for the remainder of the paper: the quasicategory of spectra will be denoted by $\S$ and the quasicategory of Kan complexes, sometimes called spaces, will be denoted $\T$; the quasicategory of small quasicategories will be denoted by $\Cat$ (to avoid set-theoretic issues we assume the existence of inaccessible cardinals as necessary, as in 1.2.15 of \cite{htt});  $\mathcal{O}^\otimes$ or $\mathcal{O}$ will always refer to an $\infty$-operad; $\E_n$ will refer to the little $n$-cubes $\infty$-operad, but sometimes when considering the $\E_\infty$-operad in its role as the terminal $\infty$-operad we will denote it by $\Fin$, to indicate that it is equivalent to the nerve of the category of finite pointed sets; for an $\E_n$-ring spectrum $R$, we denote the $\E_{n-1}$-monoidal quasicategory of left $R$-modules over $R$ as an $\E_1$-ring spectrum by $LMod_R$; $BGL_1(R)$ will be the Kan complex defined in \cite{abghr}, i.e. the delooping of the Kan complex of homotopy automorphisms of $R$ in $LMod_R$.

\section{Intermediate Thom Spectra}\label{its}

The following theorem describes our general method for producing intermediate Thom spectra:

\begin{thm}\label{mainthm}
Suppose $Y\overset{i}\to X\overset{q}\to B$ is a fiber sequence of reduced $\E_n$-monoidal Kan complexes for $n>1$ with $i$ and $q$ both maps of $\E_n$-algebras. Let $f\colon X\to BGL_1(R)$ be a morphism of $\E_n$-monoidal Kan complexes for $n>1$. Then there is a a morphism of $\E_{n-1}$-algebras $B\to BGL_1(M(f\circ i))$ whose associated Thom spectrum is equivalent to $Mf$. 
\end{thm}

The following two corollaries follow immediately from Theorem \ref{mainthm}:

\begin{cor}\label{quotiso}
Given the assumptions of Theorem \ref{mainthm}, there is an equivalence of $\E_{n-1}$-$R$-algebras $Mf\simeq M(f\circ i)\wedge_{R[\Omega B]} R$ where $R$ is equipped with the trivial $R[\Omega B]$-module structure. 
\end{cor}

\begin{proof}
For a fiber sequence $Y\to X\to B$ of $\E_n$-spaces we have a fiber sequence $\Omega B\to Y\to X$ such that $X$ is equivalent to a bar construction $Bar_\bullet(Y,\Omega B,\ast)$. Thus, since the Thom spectrum functor is symmetric monoidal and preserves colimits (cf. Corollary 8.1 of \cite{abg}, or Lewis' slightly weaker result in \cite{lewis}), the Thom spectrum of $X\to BGL_1(R)$ is equivalent to the bar construction in $\E_n$-$R$-algebras, and such in general only admits the structure of an $\E_{n-1}$-algebra.
\end{proof}

\begin{rem}
Constructing Thom spectra as bar constructions is not a new idea, and should be compared to the bar construction definition of generalized Thom spectra given in Sections 23.4 and 23.5 of \cite{maysig}. 
\end{rem}

\begin{cor}\label{thomiso}
Given the assumptions of Theorem \ref{mainthm}, there is a morphism of $\E_{n-1}$-$R$-algebra spectra $R\to M(f\circ i)\to Mf$ and a relative Thom isomorphism $Mf\wedge_{M(f\circ i)} Mf\simeq Mf\wedge_R R[B]$ where $R[B]=R\wedge_\sph\Sigma^\infty_+ B$. 
\end{cor}

\begin{proof}
The fact that the equivalence exists and is an equivalence of $\E_{n-1}$-$R$-algebras follows from Corollary 1.8 of \cite{abg}. In particular, we know that the equivalence is given by a morphism $Mf\wedge_{M(f\circ i)} Mf\to Mf\wedge_{M(f\circ i)}Mf\wedge_R R[B]\to Mf\wedge_R R[B]$, where the first map is the Thom diagonal and the second map is the $M(f\circ i)$-algebra structure map of $Mf$.
\end{proof}

We now give a proof of Theorem \ref{mainthm}, though it relies on Lemmas which we defer to Section \ref{lemmas}. It also makes crucial use of the notion of an \emph{operadic left Kan extension}, as described in Section 3.1.2 of \cite{ha}.

\begin{proof}
Note that $M(f\circ i)$ is an $\E_n$-algebra, so $BGL_1(M(f\circ i))$ is an $(n-1)$-fold loop space, so we cannot hope for the desired map to be more structured than this. By Lemmas \ref{kanextexist} and \ref{kanextcomp} the $\E_{n-1}$-monoidal left Kan extension of $X\overset{f}\to BGL_1(\sph)\hookrightarrow \S$ along $q\colon X\to B$ exists and takes the unique 0-simplex of $B$ to the $\E_n$-algebra $M(f\circ i)$. By Proposition \ref{factorization}, this Kan extension factors as a morphism of $\E_{n-1}$-monoidal Kan complexes through $BGL_1(M(f\circ i))$. Taking the Thom spectrum of the induced morphism $B\to BGL_1(M(f\circ i))$ produces $M(f\circ i)/(\Omega B)$ as a Thom spectrum over $M(f\circ i)$. Moreover, taking the colimit of the functor $B\to BGL_1(M(f\circ i))\hookrightarrow LMod_{M(f\circ i)}$ is equivalent to taking the colimit of the underlying spectra, by Corollary 4.2.3.7 of \cite{ha}. However, taking the colimit in spectra is equivalent to forming the left operadic Kan extension of $B\to\S$ along the map $B\to \ast$. By Lemma \ref{flatfibration} and Corollary 3.1.4.2 of \cite{ha} we have that the left operadic Kan extension along $X\to B$ followed by the left operadic Kan extension along $B\to \ast$ is equivalent to the left operadic Kan extension along $X\to \ast$ (i.e. Kan extensions compose). In other words, the $\E_{n-1}$-$M(f\circ i)$-module $M(f\circ i)/(\Omega B)$ has an underlying spectrum equivalent to the colimit of $X\to BGL_1(\sph)$ which is of course $Mf$. Thus the iterated Kan extension which produces $M(f\circ i)=\sph/\Omega Y$ and then quotients it by the action of $\Omega B$ is equivalent to the one-step Kan extension producing $\sph/\Omega X\simeq Mf$ with an ``action" of the trivial $\E_{n-1}$-space. Hence $Mf$ is produced as a Thom spectrum over $M(f\circ i)$. 
\end{proof}

\subsection{The Lemmas}\label{lemmas}

\begin{lem}\label{kanextexist}
Let $X$ be a Kan complex and $f\colon X\to \mathcal{C}$ an $\E_n$-monoidal morphism of quasicategories where $\mathcal{C}$ is a cocomplete quasicategory. Then for any morphism of $\E_n$-monoidal Kan complexes $p\colon X\to B$, the operadic Kan extension of $f$ along $p$ exists.
\end{lem}

\begin{proof}
Since $X$ and $B$ are Kan complexes, hence essentially small, and $\mathcal{C}$ is cocomplete, the result follows from Corollary 3.1.3.5 of \cite{ha}. 
\end{proof}

\begin{lem}\label{kanextcomp}
Let $Y\overset{i}\to X\overset{q}\to B$ be a fiber sequence of $\E_n$-monoidal Kan complexes. The $\E_{n}$-monoidal left Kan extension of an $\E_n$-monoidal morphism $f\colon X\to BGL_1(\sph)\hookrightarrow \S$ along $q\colon X\to B$ is computed by taking the colimit of the composition $$fib(X\to B)\simeq Y\to X\to BGL_1(\sph)\hookrightarrow \S.$$
\end{lem}

\begin{proof}
Following the notation given in Definition 3.1.2.2 and the construction in Remark 3.1.3.15 of \cite{ha}, we have a correspondence of $\infty$-operads given by $$\mathcal{M}^\otimes\simeq (X^\otimes\times\Delta^1)\coprod_{X^\otimes\times\{1\}} B^\otimes\to\Fin\times\Delta^1.$$ In other words, there is a family of $\infty$-operads indexed by $\Delta^1$ which looks like $X^\otimes$ (the $\infty$-operad associated to $X$ as an $\E_n$-monoidal Kan complex) at one end and $B^\otimes$ at the other end.  Formula $(\ast)$ of Definition 3.1.2.2 of \cite{ha} states that the value of the desired Kan extension at a 0-simplex $\sigma\in B$ is given by the colimit diagram:

$$ ((\mathcal{M}_{act}^\otimes)_{/\sigma}\times_{\mathcal{M}^\otimes}X^\otimes)^{\triangleright}\to(\mathcal{M}^\otimes)^{\triangleright}_{/\sigma}\to \mathcal{M}^\otimes\to\T$$
 where the morphism $(\mathcal{M}^\otimes)^{\triangleright}_{/\sigma}\to \mathcal{M}^\otimes$ takes the cone point to $\sigma$. In other words, the value of the Kan extension at $\sigma$ is computed by taking the colimit over the diagram in $\mathcal{M}^\otimes$ of objects (and active morphisms) living over $\sigma$. As the simplicial set $\mathcal{M}^\otimes$ is nothing more than the mapping cylinder of the morphism of $\E_n$-monoidal Kan complexes $X^\otimes\to B^\otimes$, we have the result. 
\end{proof}

\begin{lem}\label{flatfibration}
There is a $\Delta^2$-family of $\infty$-operads induced by the morphisms of $\E_{n-1}$-monoidal Kan complexes $X\to B$ and $B\to \ast$, denoted $\mathcal{M}^\otimes\to \Delta^2\times\Fin$, and the induced projection $\mathcal{M}^\otimes\to \Delta^2$ is a flat categorical fibration. 
\end{lem}

\begin{proof}
The equivalence of morphisms $(X\to B\to \ast)\simeq (X\to \ast)$ is given by a 2-simplex in the quasicategory of $\E_{n-1}$-monoidal quasicategories, hence by a morphism of simplicial sets in $Hom(\Delta^2,Hom(\E_{n-1}^\otimes,\Cat))\simeq Hom(\Delta^2\times\E_{n-1}^\otimes,\Cat).$ By the quasicategorical Grothendieck construction of \cite{htt}, we obtain a coCartesian fibration of simplicial sets $p\colon\mathcal{M}^\otimes\to \Delta^2\times \E_{n-1}^\otimes$ such that $p^{-1}(0)\simeq X^\otimes$, $p^{-1}(1)\simeq B^\otimes$ and $p^{-1}(2)\simeq \ast^\otimes$, where $X^\otimes$, $B^\otimes$ and $\ast^\otimes$ are the $\infty$-operads witnessing the $\E_{n-1}$-monoidal structure on $X$, $B$ and $\ast$. The projection map induces a family of $\infty$-operads $\mathcal{M}^\otimes\to \Delta^2$. This projection is a flat fibration as it satisfies the requirements of Example B.3.4 of \cite{ha}, i.e. there are coCartesian lifts of every edge in $\Delta^2\simeq\Delta^2\times\ast\subset\Delta^2\times\Fin$. 
\end{proof}

\begin{prop}\label{factorization}
Let $Y\overset{i}\to X\overset{q}\to B$ be a fiber sequence of reduced, connected $\E_n$-monoidal Kan complexes. The left operadic Kan extension of an $\E_n$-morphism $f\colon X\to BGL_1(\sph)\to \S$ along the $\E_n$-morphism $q\colon X\to B$ factors as a morphism of $\E_{n-1}$-monoidal Kan complexes through $BGL_1(M(f\circ i))$.
\end{prop}

\begin{proof}
Note that the left operadic Kan extension along $q$ takes the unique zero simplex of $B$ to $M(f\circ i)$ by Lemma \ref{kanextcomp}. Since $B$ is an $\E_n$-monoidal Kan complex it is also an $\E_n$-monoidal quasicategory with monoidal unit $1_B$ corresponding to the base point of $B$. Moreover all the morphisms of $B$ are also $1_B$-module isomorphisms. In other words, $LMod_{1_B}\simeq BGL_1(1_B)\simeq B$ as $\E_{n-1}$-monoidal quasicategories (also cf.~Corollary 4.2.4.9 of \cite{ha}). Hence it  must be that this Kan extension, being an $\E_n$-monoidal functor, induces an $\E_{n-1}$-monoidal functor $BGL_1(1_B)\simeq B\to BGL_1(M(f\circ i))$. 
\end{proof}

\begin{rem}
We can think of the identification $B\simeq BGL_1(1_B)$ as a construction of the delooping of $\Omega B$ by taking the base point component of $Pic(LMod_{\Omega B})$. In other words as a quasicategory $B$ can be thought of as the maximal $\E_{n-1}$-monoidal Kan complex on the object $\Omega B\in LMod_{\Omega B}$. 
\end{rem}

\section{Examples}\label{interexs}

A large number of morphisms of $\E_n$-monoidal Kan complexes  fit into the framework described in the introduction and Theorem \ref{mainthm}. In the following we repeatedly use the fact from 6.39 in \cite{boardvogt} that there is a sequence of infinite loop maps $U\to O\to GL_1(\sph)$ (where they write $F$ for $GL_1(\sph)$). The delooped (infinite loop) map $BO\to BGL_1(\sph)$ is called the $j$-homomorphism, and  the composition $BU\to BO\to BGL_1(\sph)$ is called the complex $j$-homomorphism.  We also use the fact that deloopings and connective covers (modeled by a bar construction and based loops on a bar construction respectively) take $\E_n$-spaces to $\E_{n-1}$-spaces and $\E_n$-spaces to $\E_n$-spaces, respectively. 

\begin{enumerate}
\item $\BSU\to \BU\to\mathbb{C}P^\infty$ is a fiber sequence of infinite loop spaces. The complex $j$-homomorphism $BU\to BGL_1(\sph)$ is a morphism of infinite loop spaces.
\item $\BString\to \BSpin\to K(\ints,4)$ is a fiber sequence of infinite loop spaces. Using the covering map $\BSpin\to BO$ composed with the $j$-homomorphism, we obtain a map of infinite loop spaces $\BSpin\to BGL_1(\sph)$.
\item $\BU\to \BSO\to \Spin$ is a fiber sequence of infinite loop spaces as a result of Table 2.1.1 of \cite{klw}, and the map $\BSO\to BGL_1(\sph)$ comes from the classical $j$-homomorphism, as above. 
\item That $BSp\to BU\to SO/U$ is a fiber sequence of infinite loop spaces also follows from \cite{klw}. 
\item $ \BString\to \BU[6,\infty)\to B^3\Spin$ is a fiber sequence of infinite loop spaces, again from \cite{klw}. The map $BU[6,\infty)\to BGL_1(\sph)$ is the obvious one. 
\item $\BSpin\to \BSO\to B(\SO/\Spin)$ is clearly a fiber sequence of infinite loop spaces, and the map $BSO\to BGL_1(\sph)$ is clear.
\item $\Omega \SU(n)\to \Omega \SU(n+1)\to \Omega S^{2n+1}$ is a fiber sequence of $\E_2$-spaces, as shown in diagram 9.1.2 of \cite{rav2}. Since, by Bott periodicity, $\Omega SU\simeq BU$, there is a morphism of $\E_2$-spaces $\Omega SU(n+1)\to \Omega SU\simeq BU\to BGL_1(\sph)$. 
\item $\BSO\to \BO\to \ints/2$ is the usual fiber sequence of infinite loop spaces giving the 1-connected cover. 
\item $\Omega^2 S^3[3,\infty)\to \Omega^2S^3\to S^1$ is a fiber sequence of $\E_2$-spaces, after \cite{mahringthom}, and the morphism $\Omega^2S^3\to BGL_1(\sph)$ is also the one given there.
\end{enumerate}

Thus from Corollaries \ref{quotiso} and \ref{thomiso} we obtain the following equivalences (with respect to the numbering given above):

\begin{enumerate}
\item $\MU\simeq \MSU\wedge_{\sph[S^1]}\sph$, and $\MU\wedge_{\MSU}\MU\simeq \MU\wedge \sph[\mathbb{C}P^\infty]$.
\item $\MSpin\simeq \MString\wedge_{K(\ints,3)}\sph$ and $\MSpin\wedge_{\MString}\MSpin\simeq \MSpin\wedge\sph[K(\ints,4)]$.
\item $\MSO\simeq\MU\wedge_{\SO/\U} \sph$ and $\MSO\wedge_{\MU}\MSO\simeq \MSO\wedge \sph[\Spin]$.
\item $\MU\simeq \MSp\wedge_{\sph[U/Sp]}\sph$ and $\MU\wedge_{\MSp}\MU\simeq MU\wedge\sph[SO/U]$.
\item $\MU[6,\infty)\simeq \MString\wedge_{BB\Spin}\sph$ and $\MU[6,\infty)\wedge_{\MString} MU[6,\infty)\simeq MU[6,\infty)\wedge\sph[B^3\Spin]$.
\item $ \MSO\simeq \MSpin\wedge_{\sph[\SO/\Spin]}\sph$ and $\MSO\wedge_{\MSpin}\MSO\simeq \MSO\wedge\sph[B(\SO/\Spin)]$.
\item $X(n+1)\simeq X(n)\wedge_{\Omega^2 SU(n)} \sph$ and $X(n+1)\wedge_{X(n)}X(n+1)\simeq X(n+1)\wedge\sph[\Omega S^{2n+1}]$.
\item $\MO\simeq \MSO\wedge_{\sph[\ints/2]}\sph$ and $\MO\wedge_{\MSO}\MO\simeq\MO\wedge\sph[\mathbb{R}P^\infty]$.
\item $H\ints/2\simeq H\ints\wedge_{\sph[\ints]}\sph$ and $H\ints/2\wedge_{H\ints}H\ints/2\simeq H\ints/2\wedge\sph[S^1]$.
\end{enumerate}

\begin{rem}\label{kunnsseq}
Some of the examples given above can be verified by computations using the spectral sequence found in Theorem 6.4 of \cite{ekmm}: $$Tor^{E_\ast(R)}_{p,q}(E_\ast(M),E_\ast(N))\Rightarrow E_{p+q}(M\wedge_RN).$$ For instance, for $E=H\ints$, we can relatively easily check that $$H_\ast(X(n+1)\wedge_{X(n)}X(n+1);\ints)\cong H_\ast(X(n+1);\ints)\otimes_{\ints} H_\ast(\Omega S^{2n+1};\ints).$$ Similar computations can be made for $\MU$ over $\MSU$ as well as for the fiber sequences appearing in Bott periodicity. Much of the relevant algebra for the latter has in fact already been determined in \cite{semcartbott}. It is the author's hope that the above equivalences will be of use to homotopy theorists doing the much harder computations related to various connective covers of $\BO$. 
\end{rem}

\begin{rem}
We remark that the relative Thom isomorphisms described above can be interpreted as torsor conditions for modules over spectral algebraic group schemes. In particular, if $X$ is a Kan complex (and thus a coalgebra by the diagonal map) then we may think of an equivalence $Mf\wedge_{M(f\circ i)} Mf\simeq Mf\wedge \sph[X]$ as giving $Spec(Mf)$ the structure of a  $Spec(\sph[X])$-torsor over $Spec(M(f\circ i))$. Indeed, in terminology familiar to noncommutative geometers, many of the above examples are \emph{Hopf-Galois extensions} in the sense of Rognes \cite{rog}. We delay an investigation of this structure to future work.  
\end{rem}

\subsection{A New Construction of $MU$}\label{MUconst}
The Lazard ring, which classifies formal group laws over discrete rings, is constructed iteratively by obstruction theory, one polynomial generator at a time (cf. \cite{laz}). The spectrum $\MU$, which classifies complex oriented ring spectra, is given in \cite{rav} as the colimit of the sequence of spectra $X(n)$ described in the previous section. Moreover, the spectra $X(n)$ are strongly related to rings used to construct the Lazard ring. This naturally leads to the question of whether or not the $X(n)$ spectra, and thus $\MU$, can also be constructed by some form of obstruction or deformation theory. Theorem \ref{mainthm} and its corollaries indicate that $X(n+1)$ is a ``torsor" over $X(n)$ for the coalgebra $\sph[\Omega S^{2n+1}]$. The stable splitting of $\Omega S^{2n+1}$ then further implies that $X(n+1)$ can be thought of a twisted polynomial extension of $X(n)$ (by a polynomial algebra with a single generator in degree $2n$). 

What we show in this section is that even more is true. By invoking Theorem 4.10 of \cite{acb}, we can deduce that $X(n+1)$ is in fact a so-called \emph{versal $\E_1$-$X(n)$-algebra of characteristic $\chi_n$}, where $\chi_n$ is a class in $\pi_{2n-1}(X(n))$. This terminology, introduced in \cite{szymikprimechar}, indicates that $X(n+1)$ can be thought of as a highly structured ($\E_1$, to be specific) homotopy quotient of $X(n)$ along $\chi_n$. It is never equivalent to the simpler process of ``coning off" that class. What is true, however, is that $X(n+1)$-module structure on a spectrum (where we are thinking of $X(n+1)$ as an $\E_1$-algebra) is equivalent to an $X(n)$-module structure on that spectrum and a null-homotopy of multiplication by $\chi_n$. Moreover, it is a result of the Nilpotence Theorem of \cite{devhopsmith} that each $\chi_n$ is nilpotent for all $n$. Recalling that $\pi_{2n-1}$ is the first homotopy degree of $X(n)$ which is not either polynomial or empty, we see then that our construction of $\MU$ is given by iteratively attaching $\E_1$-cells along nilpotent elements, which is exactly what one might expect to do if one wished to construct the universal nilpotence detecting ring spectrum (which $\MU$ is). 

\begin{defi}
Given $\alpha\in\pi_k(R)$ for $R$ an $\E_n$-ring spectrum, we define the versal $\E_{n-1}$-$R$-algebra of characteristic $\alpha$ to be the following pushout in $\E_{n-1}$-$R$-algebras:
$$\xymatrix{
Fr_{\E_{n-1}}(\Sigma^{k}R)\ar[d]^{adj(\alpha)}\ar[r]^-{adj(0)} & R\ar@{-->}[d]\\
R\ar@{-->}[r] & R//\alpha
}$$
where $Fr_{\E_{n-1}}$ is the free $\E_{n-1}$-algebra functor and the maps $adj(\alpha)$ and $adj(0)$ are the adjoints of the associated maps of $R$-modules $\alpha \colon\Sigma^kR\to R$ and $0\colon\Sigma^k R\to R$.
\end{defi}

\begin{cor}
Let $X(n)$ be the Thom spectrum associated to the morphism of $\E_2$-monoidal Kan complexes $\Omega \SU(n)\to \BU\to BGL_1(\sph)$. Then $X(n+1)$ is a versal $\E_1$-algebra over $X(n)$ of characteristic $\chi_n$ where $\chi_{n}$ is a canonical class in $\pi_{2n-1}(X(n))$. 
\end{cor}

\begin{proof}
Given the fiber sequence $\Omega \SU(n)\to\Omega \SU(n+1)\to \Omega S^{2n+1}$, and an application of Theorem \ref{mainthm} above, we can identify $X(n+1)$ as the $\E_1$-monoidal Thom spectrum given by the $\E_1$-monoidal left Kan extension $\Omega S^{2n+1}\to BGL_1(X(n))$. By application of standard adjunctions, the map of $\E_1$-monoidal Kan complexes $\tilde{\tilde{\chi}}_n\in Map_{\E_1}(\Omega S^{2n+1},BGL_1(X(n)))$ induces a map of Kan complexes $\tilde{\chi}_n\in Map_{\T}(S^{2n-1},GL_1(X(n)))$. Note that $\tilde{\chi}_n$ must have image contained in a connected component $u\in\pi_0(GL_1(X(n))\simeq\ints/2$ which induces a translation $\tau_u\colon\Omega^\infty X(n)\to\Omega^\infty X(n)$. The composition $\tau_u\circ\tilde{\chi}_n\colon S^{2n-1}\to \Omega^\infty X(n)$ lifts to a morphism of spectra ${\chi}_n\colon\sph^{2n-1}\to X(n)$. An application of Theorem 4.10 of \cite{acb} gives that $X(n+1)$ is the versal $\E_1$-algebra of characteristic ${\chi}_n$ on $X(n)$. In other words, $X(n+1)$ in the following diagram is a pushout:

$$\xymatrix{
F_{\E_1}(\Sigma^{2n-1}X(n))\ar[r]^-{adj(0)}\ar[d]^-{adj({\chi}_n)}&X(n)\ar[d]\\
X(n)\ar[r]&X(n+1)
}$$
\end{proof}

\begin{rem}
The content of \cite{acb} allows us to consider $X(n+1)$ as the $\E_1$-spectrum obtained by attaching an $\E_1$-$X(n)$-cell to $X(n)$ along the map ${\chi}_n$ described above. Note that ${\chi}_1$, as a non-zero element of $\pi_1(\sph)$ must be equivalent to $\eta$, the Hopf element. The Hopf element is, of course, the first nilpotent element in the stable homotopy groups of spheres, and so again, it stands to reason that it would be the first element eliminated in an effort to construct the maximal nilpotence detecting ring spectrum. 
\end{rem} 

The following result is included since it follows immediately from \cite{acb}.

\begin{cor}
The $\E_1$-cotangent complex of the $\E_1$-algebra $X(n+1)$ in $X(n)$-modules is equivalent to $\Sigma^{2n}F_{\E_1}(X(n))\wedge_{X(n)}X(n+1).$
\end{cor}

\begin{proof}
Cf.~Proposition 5.4 of \cite{acb}. 
\end{proof}

\printbibliography

\end{document}